\documentclass[submission%
]{dmtcs}

\usepackage[latin1]{inputenc}
\usepackage{subfigure}
\usepackage[all]{xy}
\usepackage{amssymb}
\usepackage{amsgen}
\usepackage{amsmath}
\usepackage{mathrsfs}
\usepackage{amsfonts}
\usepackage{tikz}

\newtheorem{Thm}{Theorem}[section]

\newtheorem{Def}[Thm]{Definition}
\newtheorem{Cor}[Thm]{Corollary}
\newtheorem{Prop}[Thm]{Proposition}

\newcommand{\Cliq}{\mathsf{Cliq}}

\renewcommand{\to}{\longrightarrow}

\newcommand{\R}{{\mathrel{\mathscr R}}} 

\newcommand{\ov}[1]{\ensuremath{\overline {#1}}}
\newcommand{\til}[1]{\ensuremath{\widetilde {#1}}}

\newcommand{\wh}{\widehat}

\newcommand{\Ext}{\mathop{\mathrm{Ext}}\nolimits}

\newcommand{\CW}{\mathop{\Sigma}\nolimits}
\newcommand{\Inc}{\operatorname{Inc}}

\usepackage{xcolor}

\def\AAA{\mathcal A} 
\def\LLL{\mathcal L} 
\def\FFF{\mathcal F} 

\author{Stuart Margolis\addressmark{1}\thanks{Email: \email{margolis@math.biu.ac.il}.}
  \and Franco Saliola\addressmark{2}\thanks{Email: \email{saliola.franco@uqam.ca}.}
  \and Benjamin Steinberg\addressmark{3}\thanks{Email: \email{bsteinberg@ccny.cuny.edu}.}}
\title{Poset topology and homological invariants of algebras arising in algebraic combinatorics}
\address{%
  \addressmark{1}Department of Mathematics, Bar Ilan University, Israel\\
  \addressmark{2}Laboratoire de Combinatoire et d'informatique Mathématique (LaCIM), UQAM, Canada\\
  \addressmark{3}City College of New York, USA}
\received{2013-11-18}
\revised{\today}
\accepted{tomorrow}

\begin{document}
\maketitle

\begin{abstract}
\paragraph{Abstract.}
We present a beautiful interplay between combinatorial topology and homological
algebra for a class of monoids that arise naturally in algebraic combinatorics.
We explore several applications of this interplay.
For instance, we provide a new interpretation of the Leray number of a clique complex in terms of non-commutative algebra.

\paragraph{Résumé.}
Nous présentons une magnifique interaction entre la topologie combinatoire et
l'algèbre homologique d'une classe de monoïdes qui figurent naturellement dans
la combinatoire algébrique.
Nous explorons plusieurs applications de cette interaction. Par exemple, nous
introduisons une nouvelle interprétation du nombre de Leray d'un complexe de
clique en termes de la dimension globale d'une certaine algèbre non
commutative.
\end{abstract}

\section{Introduction}
\label{sec:in}
In a highly influential paper~\cite{BHR}, Bidigare, Hanlon and Rockmore showed
that a number of popular Markov chains are random walks on the faces of
a hyperplane arrangement. Their analysis of these Markov chains took advantage
of the monoid structure on the set of faces. This theory was later extended by
Brown~\cite{Brown1,Brown2} to a larger class of monoids called left regular
bands. In both cases, the representation theory of these monoids play
a prominent role. It is used to compute the spectrum of the transition
operators of the Markov chains and to prove diagonalizability of the transition
operators.

It develops that there is a close connection between algebraic and
combinatorial invariants of these monoids: certain homological invariants of
the monoid algebra coincide with the cohomology of order complexes of posets
naturally associated with a monoid. We present here a synopsis of the results
together with applications of this beautiful interplay between combinatorial
topology and homological algebra.

The full details will appear in two longer papers~\cite{oldpaper, newpaper},
which include lengthy discussions of the history, examples and applications of
these monoids as well as thorough references to the literature.

\section{Left Regular Bands}
\label{s:LRBs}

A particularly important class of monoids arising in probability theory and in
algebraic combinatorics is the class of left regular
bands~\cite{Brown1,Brown2,DiaconisBrown1,Saliola,Saliolahyperplane,bjorner2,GrahamLRB}.
\begin{Def}
    \label{DefLRB1}
A \emph{left regular band} is a semigroup $B$ satisfying the identities
\begin{gather}\label{leftregularity}
    x^2=x\ \text{for all}\ x\in B,\\
    xyx = xy\ \text{for all}\ x,y \in B.
\end{gather}
In general we will assume in this paper that $B$ is finite and has an identity (is a monoid).
\end{Def}

An example arises from the set $\{0, +, -\}$, where the monoid operation
$\circ$ satisfies $x \circ y = x$ if $x\neq0$ and $x \circ y = y$ otherwise.
Several important examples of left
regular bands arising in combinatorics are submonoids $\{0,+,-\}^n$.
These include the monoid structure on both real hyperplane
arrangements, matroids, oriented matriods and interval greedoids
(see Section~\ref{s:LRBExamples}).

It is known that left regular bands are precisely the homomorphic images of
submonoids of $\{0, +, -\}^n$, for some $n$.
The paper \cite{qvariety} classifies those that are submonoids of some
$\{0, +, -\}^n$.
Almost all the left regular bands appearing so far in the algebraic
combinatorics literature embed in $\{0,+,-\}^n$; the main exception are the
complex hyperplane monoids discussed in \S~\ref{ss:complexhyperplanemonoids}
below.

The following two posets associated with a left regular band play a central
role in what follows.

\paragraph{$\R$-poset.}
For a monoid $M$ and $m \in M$, we let $mM = \{mm' : m' \in M\}$ denote the
\emph{principal right ideal} generated by $m$.
\emph{Green's $\R$-preorder} is defined on $M$ by
$m\leq_\R n$ if $mM\subseteq nM$.
The associated equivalence relation is denoted $\R$; it is one of
\emph{Green's relations} on a monoid~\cite{Green}.

In a left regular band $B$, Green's $\R$-preorder admits the following
description:
\begin{gather*}
    a \leq_\R b \text{~~iff~~} ba = a.
\end{gather*}
Thus, if $a \leq_\R b$ and $b \leq_\R a$, then $a=aba=ab=b$.
It follows that $B$ is a poset with respect to $\leq_\R$.
We call this partial order the \emph{$\R$-order} on $B$ and denote it simply by
$\leq$.
Figures~\ref{figure: free lrb} and~\ref{fig:3Lines}
illustrate the $\R$-order on specific examples.

\paragraph{Support Lattice.}

For a left regular band $B$, let $\Lambda(B)$ denote the set
of principal left ideals of $B$:
\begin{gather*}
    \Lambda(B) = \left\{ Ba : a \in B \right\}
\end{gather*}
where $Ba = \{ba : b \in B\}$. Then $\Lambda(B)$ is a poset under inclusion.
Furthermore, it is a lattice with intersection as the meet operation (denoted $\wedge$).
The map $\sigma\colon B \xrightarrow{} \Lambda(B)$ given by $\sigma(a)=Ba$
satisfies the following properties for all $a,b \in B$:
\begin{gather}
\sigma(a b) = \sigma(a) \wedge \sigma(b)
\qquad\qquad\text{and}\qquad\qquad
a b = a \text{~~iff~~}\sigma(b) \geq \sigma(a)
\end{gather}
The first property says that $\sigma$ is a monoid morphism.
Note that $\sigma$ is an order preserving poset morphism:
that is, if $a \leq_\R b$, then $\sigma(a) \leq \sigma(b)$.
Proofs of these statements can be found in~\cite{Clifford, Brown1}.

Following the standard conventions of lattice theory, we denote by $\wh 1$ the
top of $\Lambda(B)$ (which is $B$, itself) and by $\wh 0$ the bottom (called
the \emph{minimal ideal} of $B$).
In~\cite{Brown1,Brown2}, Brown calls $\Lambda(B)$ the \emph{support lattice} of
$B$ (although he uses the reverse ordering) and $\sigma$ is called the
\emph{support map}.

\section{Examples of Left Regular Bands}
\label{s:LRBExamples}
This section quickly surveys some examples of left regular bands found in the
combinatorics literature. Often, these are special cases of certain
semigroup-theoretic constructions.
Further examples and detailed expositions can be found in~\cite{oldpaper}.
We also describe a new class of left regular bands introduced in
\cite{oldpaper}.

\paragraph{Free left regular bands.}
The \emph{free left regular band} on a set $A$ will be denoted $F(A)$ and the
free left regular band on $n$-generators will be written $F_n$.
One can view $F(A)$ as the set of words over $A$ with no repeated letters;
multiplication is given by concatenation followed by removal of repetitions
(reading from left to right).
For example, in $F_5$ we have $523 \cdot 13245 = 52314$.
The support lattice of $F(A)$ can be identified with the power set $P(A)$ with
the operation of union; the support map $\sigma$ takes a word to the set of
letters appearing in the word;
Figure~\ref{figure: free lrb} illustrates
the $\R$-order and the support lattice of
$F(\{a,b,c\})$.

\begin{figure}[htb]
\begin{gather*}
\xymatrix@l@C=0pt@R=1em{
cba \ar@{-}[d] && cab \ar@{-}[d] && bca \ar@{-}[d] && bac \ar@{-}[d] & & acb \ar@{-}[d] && abc \ar@{-}[d] \\
 cb \ar@{-}[dr] &&  ca \ar@{-}[dl] &&  bc \ar@{-}[dr] &&  ba \ar@{-}[dl] & &  ac \ar@{-}[dr] &&  ab \ar@{-}[dl] \\
    &c \ar@{-}[drrrr] &    &&    &b\ar@{-}[d]&     & &    &a\ar@{-}[dllll]& \\
    &&     &&    &1&
}
\hspace{3em}
\xymatrix@l@C=1em@R=1em{
& \{a,b,c\} \ar@{-}[dl] \ar@{-}[rd] \ar@{-}[d] \\
\{b,c\} \ar@{-}[d]\ar@{-}[dr] & \{a,c\} \ar@{-}[dl]\ar@{-}[dr] & \{a,b\} \ar@{-}[d]\ar@{-}[dl] \\
\{c\} \ar@{-}[dr]  & \{b\} \ar@{-}[d]  & \{a\} \ar@{-}[dl] \\
 & \emptyset
}
\end{gather*}
\caption{The $\R$-order and the support lattice of $F(\{a,b,c\})$.}
\label{figure: free lrb}
\end{figure}
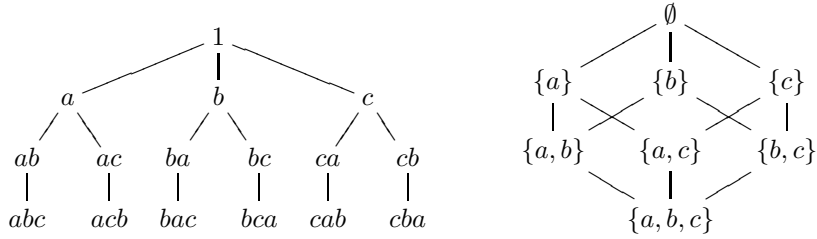

\paragraph{Hyperplane face monoids and oriented matroids}

Hyperplane arrangements, and more generally oriented matroids,
provide an important source of examples of left regular bands,
which turn out to be submonoids of the left regular band $\{0,+,-\}^n$
introduced in Section~\ref{s:LRBs}.
We recall the construction and properties of these left regular bands referring
the reader to~\cite[Appendix~A]{Brown1} for details.

A \emph{central hyperplane arrangement} in $V = \mathbb R^n$ is a finite
collection $\AAA$ of hyperplanes of $V$ passing through the origin. For each
hyperplane $H \in \AAA$, fix a labelling $H^+$ and $H^-$ of the two open half
spaces of $V$ determined by $H$; the choice of labels $H^+$ and $H^-$ is
arbitrary, but fixed throughout. Let $H^0 = H$.

The elements of the monoid are the \emph{faces} of $\AAA$: the non-empty
intersections of the form
$
x = \bigcap_{H \in \AAA} H^{\epsilon_H}
$
with $\epsilon_H \in \{0,+,-\}$.
Thus, a face $x$ is completely determined by the sequence
$\varepsilon(x) = (\epsilon_H)_{H\in\AAA}$.
See Figure~\ref{fig:3Lines}.
The image of $\varepsilon$ identifies the set of faces of $\AAA$
with a submonoid of $\{0,+,-\}^{\AAA}$ and so we
obtain a monoid structure by defining the product of $x$ and $y$
to be the face with sign sequence $\varepsilon(x) \circ \varepsilon(y)$.


The monoid of faces of $\AAA$ is called the
\emph{face monoid} and is denoted $\FFF$.
The lattice $\Lambda(\FFF)$ is isomorphic to
the \emph{intersection lattice} $\LLL$ of $\AAA$; it consists of the subspaces
of $V$ that can be expressed as an intersection of hyperplanes from $\AAA$.
Under this isomorphism, the support map $\sigma$
corresponds to the map that sends a face $x$ to the smallest subspace that
contains $x$.

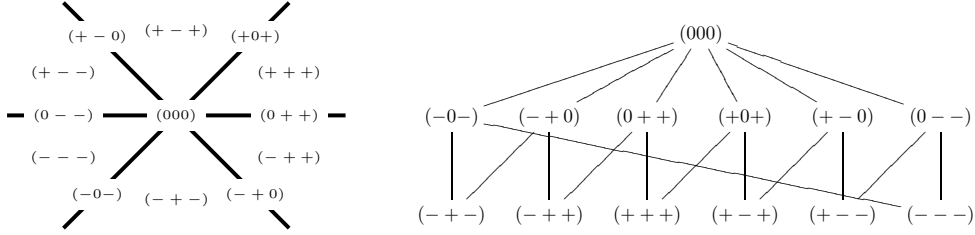
\begin{figure}[htb]
\centering
\begin{tikzpicture}[scale=0.75]
    \draw[line width=1.5pt] (-2,-2) -- (2,2);
    \draw[line width=1.5pt] (2,-2) -- (-2,2);
    \draw[line width=1.5pt] (-3,0) -- (3,0);
    \draw (0,0)       node[fill=white,font=\tiny] {$(000)$};
    \draw (2,0.75)    node[font=\tiny]           {$(+++)$};
    \draw (2,0)       node[fill=white,font=\tiny] {$(0++)$};
    \draw (2,-0.75)   node[font=\tiny]           {$(-++)$};
    \draw (1.4,-1.4)  node[fill=white,font=\tiny] {$(-+0)$};
    \draw (0,-1.5)    node[font=\tiny]           {$(-+-)$};
    \draw (-1.4,-1.4) node[fill=white,font=\tiny] {$(-0-)$};
    \draw (-2,-0.75)  node[font=\tiny]           {$(---)$};
    \draw (-2,0)      node[fill=white,font=\tiny] {$(0--)$};
    \draw (-2,0.75)   node[font=\tiny]           {$(+--)$};
    \draw (-1.4,1.4)  node[fill=white,font=\tiny] {$(+-0)$};
    \draw (0,1.5)     node[font=\tiny]           {$(+-+)$};
    \draw (1.4,1.4)   node[fill=white,font=\tiny] {$(+0+)$};
\end{tikzpicture}
\qquad
$
\resizebox{3in}{!}{
\xygraph{
!{0;/r18mm/:,p+/u18mm/::}
!~:{@{-}}
{(---)}(
    :[u]{(0--)}
    :[dl]{(+--)}
    :[u]{(+-0)}
    :[dl]{(+-+)}
    :[u]{(+0+)}
    :[dl]{(+++)}
    :[u]{(0++)}
    :[dl]{(-++)}
    :[u]{(-+0)}
    :[dl]{(-+-)}
    :[u]{(-0-)}
    :"(---)"),
"{(-0-)}":[rrru(.85)]{(000)}(:"(-+0)",:"{(0++)}",:"(+0+)",:"(+-0)",:"(0--)")
 )}
}
$
\caption{Sign sequences of the faces of the arrangement in
$\mathbb R^2$ of three distinct lines
and the associated $\R$-order.}
\label{fig:3Lines}
\end{figure}

\paragraph{Free partially commutative left regular bands.}
If $\Gamma=(V,E)$ is a simple graph, then the \emph{free partially commutative left
regular band} associated with $\Gamma$ is the left regular band $B(\Gamma)$ with
presentation
\begin{align*}
    B(\Gamma) = \big\langle V \mid xy=yx \text{~for all edges~} \{x,y\}\in E\big\rangle.
\end{align*}
This is the left regular band analogue of free partially commutative monoids
(also called \emph{trace monoids} or \emph{graph
monoids}~\cite{tracebook,CartierFoata}) and of free partially commutative
groups (also called \emph{right-angled Artin groups} or \emph{graph
groups}~\cite{Wise}).
For example, if $\Gamma$ has no edges,
then $B(\Gamma)$ is the free left regular band on the vertex set,
whereas if $\Gamma$ is a complete graph, then $B(\Gamma)$ is the free
semilattice on the vertex set of $\Gamma$.

The elements of $B(\Gamma)$ correspond to acyclic
orientations of induced subgraphs of the complementary graph $\ov \Gamma$ of
$\Gamma$.
Indeed, any element $w$ of the free left regular band $F(V)$,
which is a repetition-free sequence of vertices,
gives rise to an acyclic orientation $\mathcal O(w)$ as follows:
the vertices of the subgraph are the letters appearing in $w$;
there is an edge $x \to y$ if $x$ comes before $y$ in $w$
and if $\{x,y\}$ is not an edge of $\Gamma$.

\begin{Thm}\label{wordproblem}
Elements $v,w\in F(V)$ are equal in $B(\Gamma)$ if and only if
$\sigma(v)=\sigma(w)$ and $\mathcal O(v)=\mathcal O(w)$.
\end{Thm}

The Tsetlin library Markov chain can be modelled as a hyperplane random walk,
but is most naturally a random walk on the free left regular band, see~\cite{Brown1,Brown2}.
Similarly, the random walk on acyclic orientations of a graph considered by
Athanasiadis and Diaconis as a function of a hyperplane
walk~\cite{DiaconisAthan} is most naturally a random walk on a free partially
commutative left regular band.

\paragraph{Other examples.}

There are numerous other examples of left regular bands appearing in disguise
in the algebraic combinatorics literature and elsewhere, as well as numerous
examples of left regular band \emph{semigroups}. These include the face
semigroup of an affine hyperplane arrangement, the set of covectors
associated with an affine oriented matroid or a $T$-convex set of topes and the
face semigroup of a finite $\mathrm{CAT}(0)$ cube complex. These are presented
in detail in~\cite{oldpaper, newpaper}.

\section{Statement of the Main Result}
\label{s:mainresult}

Our main result expresses certain algebraic invariants of $B$ in terms of the
cohomology of simplicial complexes associated with its $\R$-poset.
This allows us to compute these algebraic invariants using the combinatorics of
the poset, which we do for certain classes of left regular bands in the next
section.


We begin by describing the algebraic invariants.
Fix a commutative ring with unit $\Bbbk$.
For every element $X$ of the support lattice $\Lambda(B)$, the ring $\Bbbk$
becomes a $\Bbbk B$-module via the action
\begin{align*}
    b \cdot \alpha =
    \begin{cases}
        \alpha, & \text{if}\ \sigma(b)\geq X,\\
        0, & \text{otherwise},
    \end{cases}
\end{align*}
for all $b \in B$ and $\alpha \in \Bbbk$.
These modules, denoted by $\Bbbk_X$, are precisely the simple $\Bbbk B$-modules when
$\Bbbk$ is a field~\cite{Brown1, Saliola}.
%
We are interested in computing the algebraic invariants
$\Ext^n_{\Bbbk B}(\Bbbk_X,\Bbbk_Y)$ for two such modules~\cite{assem}.
It turns out that these invariants coincide precisely with
the cohomology of certain simplicial complexes associated with $B$,
so a detailed description of $\Ext^n$ is not necessary for what follows.
Indeed, the power of the main result resides in the fact that we can compute
these algebraic invariants by making use of combinatorial properties of $B$.

For elements $X$ and $Y$ of the support lattice $\Lambda(B)$,
pick an element $y \in B$ with $\sigma(y) = Y$ and define
\begin{gather*}
    B[X,Y] = \big\{ z \in B : X \leq \sigma(z) \text{~and~} z \leq_\R y \big\}.
\end{gather*}
Passing from $B$ to the monoid $B[X,Y]$ corresponds to the
restriction--contraction operation for oriented matroids.
This is a subposet of $B$ with respect to Green's $\R$-order. Let $\Delta(X,Y)$
be the \emph{order complex} of the poset $B[X,Y] \setminus \{y\}$:
it is the simplicial complex whose vertex set is $B[X,Y] \setminus \{y\}$ and
whose simplices are the finite chains in the poset. (Up to isomorphism, this poset does not depend on the choice of $y\in Y$.)

\begin{Thm}\label{mainresult}
Let $B$ be a finite left regular band,
$\Bbbk$ a commutative ring with unit,
and $X,Y\in \Lambda(B)$.
Then
\begin{align*}
    \Ext^n_{\Bbbk B}(\Bbbk_X,\Bbbk_Y) =
    \begin{cases}
        \til H^{n-1}(\Delta(X,Y),\Bbbk), & \text{if}\ X<Y,\ n\geq 1,\\
        \Bbbk, &\text{if}\ X=Y,\ n=0,\\
        0,  & \text{otherwise.}
    \end{cases}
\end{align*}
\end{Thm}

\section{Applications of the Main Result}
\label{s:applications}


\subsection{Quiver with relations of a left regular band algebra}

We begin by investigating applications of the degree $0$ and degree $1$
cohomology.
They encode information about the structure of the monoid algebra $\Bbbk B$
and allows us to describe a \emph{quiver presentation} of $\Bbbk B$.

\begin{figure}[htb]
    \centering
    $\vcenter{
    \xymatrix@l@C=0pt@R=1em{
    cba \ar@{-}[d] && cab \ar@{-}[d] && bca \ar@{-}[d] && bac \ar@{-}[d] & & acb \ar@{-}[d] && abc \ar@{-}[d] \\
     cb \ar@{-}[dr] &&  ca \ar@{-}[dl] &&  bc \ar@{-}[dr] &&  ba \ar@{-}[dl] & &  ac \ar@{-}[dr] &&  ab \ar@{-}[dl] \\
        &c \ar@{-}[drrrr] &    &&    &b\ar@{-}[d]&     & &    &a\ar@{-}[dllll]& \\
        &&     &&    &1&
    }}$
    \hspace{.5in}
    $\vcenter{
    \xymatrix@l@M=1pt{
      & \{a,b,c\}\ar@/^7ex/[ddd] \ar@/_7ex/[ddd] \ar@/_1.5ex/[ddl] \ar@/^1.5ex/[rdd]
    \ar@<-0.4ex>@/^4ex/[dd] \\
    \{a,b\} \ar@/_2ex/[ddr] & \ \{a,c\} \ar@<0.4ex>@/_4ex/[dd] & \{b,c\} \ar@/^2ex/[ddl] \\
    \{a\} & \{b\} \ & \{c\} \\
     & \emptyset
    }}
    $
\caption{The $\R$-order on $F(\{a,b,c\})$ and the quiver of $\Bbbk F(\{a,b,c\})$.}
\label{fig:F3quiver}
\end{figure}
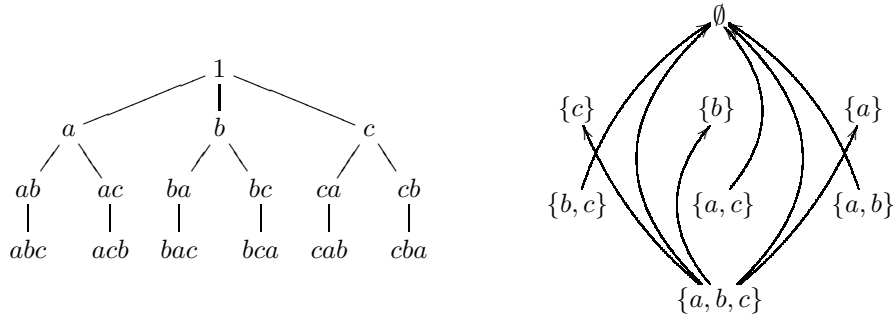

\paragraph{Quiver.}
Assume for the moment that $\Bbbk$ is a field.
Then the dimension of the degree $0$ cohomology of $\Delta(X,Y)$
counts its number of connected components~\cite{Wachs}.
Combined with Theorem~\ref{mainresult}, we obtain a combinatorial
interpretation of the dimension of $\Ext^1_{\Bbbk B}(\Bbbk_X, \Bbbk_Y)$.
The dimensions of these spaces carry all the information needed to construct
the quiver of an algebra~\cite{assem}, which yields the following new
description for the quiver of $\Bbbk B$. This is more conceptual and
sometimes easier to apply than~\cite{Saliola}.

\begin{Thm}\label{quiver}
Let $B$ be a finite left regular band and $\Bbbk$ a field.  Then the quiver of
$\Bbbk B$ has vertex set $\Lambda(B)$.  The number of arrows from $X$ to $Y$ is
zero unless $X<Y$, in which case it is one less than the number of connected
components of $\Delta(X,Y)$.
\end{Thm}

The quiver of $\Bbbk F_3$ is depicted in Figure~\ref{fig:F3quiver};
for instance, there are two arrows from the bottom to the top since
there are three connected components in the Hasse diagram of the
$\R$-order on $F_3 \setminus \{1\}$.
Corollary~\ref{freelrb} provides a description of the quiver of
$F_n$ for all $n \geq 0$.

\paragraph{Quiver relations.}

One can think of a quiver as a combinatorial encoding of a special
presentation of the algebra by generators and relations: the quiver is
a directed graph whose vertices correspond to primitive orthogonal idempotents
and whose arrows correspond to a minimal generating set of the radical.
Relations take the form of linear combinations of paths of the quiver~\cite{assem}.
A result of Bongartz~\cite{assem,Bongartz} implies that the number of relations
involving paths between two vertices in a minimal quiver presentation
for a left regular band algebra is determined by the dimension of $\Ext^2$.
Theorem~\ref{mainresult} admits the following corollary.

\begin{Cor}
    \label{cor:minrelations}
    If $X<Y$ in $\Lambda(B)$, then the number of relations involving paths from
    $X$ to $Y$ in a minimal quiver presentation of a left regular band algebra
    $\Bbbk B$ is given by $\dim_{\Bbbk} \til H^1(\Delta(X,Y),\Bbbk)$.
\end{Cor}


\subsection{Algebras of free left regular bands are hereditary}
\label{ss:righthereditary}

A finite dimensional algebra whose left ideals are projective modules is called
\emph{hereditary}.
This is equivalent to the property that $\Ext^2$ vanishes;
to the property that each submodule of a projective module is projective;
and to the property that the global dimension of the algebra is at most $1$. A result of Gabriel says that a split basic algebra is hereditary precisely when it has a quiver presentation with an acyclic quiver and no relations~\cite{assem}.
In section~\ref{ss:gldim} we will see a characterization of those left regular band algebras
$\Bbbk B$ that are hereditary in terms of the order complex of $B$, but we
foreshadow it by showing that the algebra of the free left regular band is
hereditary.

The Hasse diagram of the $\R$-order of the free left regular
band $F_n$ is a tree, as are the diagrams of the subposets $F_n[X,Y]$.
Consequently, the cohomology of the order complexes $\Delta(X,Y)$ vanishes in
positive degrees.
Theorem~\ref{quiver} and Corollary~\ref{cor:minrelations}
combine to show that $\Bbbk F_n$ is hereditary
and to describe its quiver.

\begin{Cor}\label{freelrb}
The algebra $\Bbbk F_n$ is hereditary over any field $\Bbbk$. Its quiver has
vertex set the subsets of $\{1,\ldots,n\}$.  If $X\supsetneq Y$, then there are
$|X\setminus Y|-1$ arrows from $X$ to $Y$.  There are no other arrows.
\end{Cor}

This result was first proved by K.~Brown using quivers and a counting
argument~\cite[Theorem~13.1]{Saliola}.
Our argument for heredity generalizes to any left regular band whose $\R$-order
is a tree; we call such left regular bands \emph{right hereditary}.
These include
the matroid left regular bands of Brown~\cite{Brown1},
the interval greedoid left regular bands of Bj\"orner~\cite{bjorner2},
and the Rhodes and Karnofsky-Rhodes expansions of a lattice~\cite{oldpaper}.
We will see in Theorem~\ref{treethm} another, more transparent, proof of
Corollary~\ref{freelrb}, which also makes it simple to compute the quiver of
a right hereditary left regular band.


\subsection{Hyperplane face monoids}
\label{ss:hyperplanemonoids}

Recall that $\AAA$ denotes a central hyperplane arrangement in
a $d$-dimensional real vector space $V$, $\LLL$ its intersection lattice, and
$\FFF$ its left regular band of faces.
Without loss of generality, we suppose that the intersection of all the
hyperplanes in $\AAA$ is the origin: otherwise quotient $V$ by this
intersection.

We argue that $\Delta(\wh0,\wh1)$ is a $(d-1)$-sphere. The $\R$-order on $\FFF$
can be described geometrically as $y \leq x$ if and only if $x \subseteq
\overline y$, where $\overline y$ denotes the set-theoretic closure of $y$.
This establishes an order-reversing bijection between the faces $\FFF$ and the
cells of the regular cell decomposition $\varSigma$ obtained by intersecting
the hyperplane arrangement with a sphere centered at the origin. The dual of
$\varSigma$ is the boundary of a polytope $Z$ so that
the poset of faces of $Z$ is isomorphic to $\FFF$
\cite{DiaconisBrown1}. Since the order complex of the poset of
faces of a polytope is the barycentric subdivision of the polytope, it follows
that $\Delta(\wh0,\wh1)$ is a $(d-1)$-sphere.

This argument also applies to $\Delta(X,Y)$ with $X \leq Y$ since it
corresponds to the hyperplane arrangement in $X$ obtained by
intersecting $X$ with the hyperplanes in $\AAA$ containing $Y$ and not $X$.
It follows that $\Delta(X,Y)$ is a sphere of dimension $\dim(X)-\dim(Y)-1$.
Consequently, we recover Lemma~8.3 of~\cite{Saliolahyperplane}.

\begin{Prop}
\label{ext-spaces}
For $X, Y \in \LLL$ and $n \geq 0$,
\[
\Ext^n_{\Bbbk\FFF}(\Bbbk_X, \Bbbk_Y) \cong
\begin{cases}
\Bbbk, & \text{if } Y \subseteq X \text{ and } \dim(X) - \dim(Y) = n, \\
0, & \text{otherwise}.
\end{cases}
\]
\end{Prop}
\begin{proof}
    We apply Theorem~\ref{mainresult}. Since $\Delta(X,Y)$ is a sphere of
    dimension $\dim(X)-\dim(Y)-1$, it follows that $\til
    H^{n-1}(\Delta(X,Y),\Bbbk)$ is $0$ unless $\dim(X)-\dim(Y)=n$,
    in which case it is $\Bbbk$.
\end{proof}

It follows that the quiver of $\Bbbk\FFF$ coincides with the Hasse diagram
of $\LLL$ ordered by reverse inclusion.
\begin{Cor}[{Saliola~\cite[Corollary~8.4]{Saliolahyperplane}}]
    \label{hyperplanequiver}
The quiver of $\Bbbk\FFF$ has vertex set $\LLL$.
The number of arrows from $X$ to $Y$ is zero unless
$Y \subsetneq X$ and $\dim(X)-\dim(Y)=1$, in which case
there is exactly one arrow.
\end{Cor}

In~\cite{Saliolahyperplane} a set of quiver relations for $\Bbbk\FFF$ was
described: for each interval of length two in $\LLL$ take the sum of all paths
of length two in the interval. It was also shown that $\Bbbk\FFF$ is a Koszul
algebra and that its Koszul dual algebra is isomorphic to the incidence algebra
of the intersection lattice $\LLL$.

\subsection{Complex hyperplane face monoids}
\label{ss:complexhyperplanemonoids}

There is a left regular band of faces that one can associated with a complex
arrangement~\cite{complexstrat} and it turns out that the situation for complex
hyperplane arrangements is similar to that of real hyperplane arrangements.
Things are slightly more complicated in this setting because the complexes
$\Delta(X, Y)$ do not arise from complex hyperplane face monoids. However,
a careful analysis of the PL structure of the cell complex associated with the
arrangement reveals that $\Delta(X,Y)$ is a sphere of dimension
$\dim_{\mathbb R} X-\dim _{\mathbb R}Y-1$, from which we obtain the following
analogues of the results from the real case,
\textit{cf.} \S~\ref{ss:hyperplanemonoids}.

\begin{Prop}
For $X, Y \in \LLL$ and $n \geq 0$,
\[
\Ext^n_{\Bbbk\FFF}(\Bbbk_X, \Bbbk_Y) \cong
\begin{cases}
\Bbbk, & \text{if } Y \subseteq X \text{ and } \dim_{\mathbb R}(X) - \dim_{\mathbb R}(Y) = n, \\
0, & \text{otherwise}.
\end{cases}
\]
\end{Prop}

An immediate consequence is that the quiver of $\Bbbk\FFF$ coincides with the Hasse diagram
of the augmented intersection lattice $\LLL$ ordered by reverse inclusion, as was the case for real hyperplane arrangements.


\subsection{Geometric left regular bands and commutation graphs}

We say that a left regular band $B$ is \emph{geometric} if,
for each $a \in B$, the left
stabilizer $\{b\in B\mid b\geq_\R a\}$ of $a$
is commutative (and hence a lattice under the order $\leq_\R$ with the meet
given by the product).
The left regular bands associated to hyperplane arrangements, oriented
matroids and $\mathrm{CAT}(0)$ cube complexes are geometric, whence the name. 
Almost all the left regular bands appearing so far in the algebraic
combinatorics literature embed in $\{0,+,-\}^n$ and hence are geometric
(see the discussion in \S~\ref{s:LRBs}).

For a finite geometric left regular band $B$, we will use the following
special case of Rota's cross-cut theorem~\cite{Rota,bjornersurvey} to provide
a simplicial complex homotopy equivalent to the order complex
$\Delta(\wh0, \wh1)$ of $B \setminus \left\{1\right\}$.
This second complex simplifies the computation of the quiver of $\Bbbk B$.

\begin{Thm}[Rota]\label{crosscutstated}
Let $P$ be a finite poset such that any subset of $P$ with a common lower bound
has a meet.  Define a simplicial complex $K$ with vertices the maximal elements
of $P$ and with simplices those subsets with a common lower bound.  Then $K$ is
homotopy equivalent to the order complex $\Delta(P)$.
\end{Thm}

We also need the well-known notion of the \emph{clique complex}
or \emph{flag complex} $\Cliq(G)$ of a simple graph $G$: it is
the simplicial complex whose vertices are the vertices of $G$ and whose
simplices are the subsets of vertices that induce a complete subgraph. Notice
that $G$ is the $1$-skeleton of $\Cliq(G)$ and that $\Cliq(G)$ is obtained by
`filling in' the $1$-skeleton of every $q$-simplex found in $G$.

Let $\mathscr M(B)$ be the set of maximal elements of $B\setminus \{1\}$ and
let $\Gamma(\mathscr M(B))$ be the \emph{commutation graph} of $\mathscr M(B)$:
that is, the graph whose vertex set is $\mathscr M(B)$ and whose edges are
pairs $(a,b)$ such that $ab=ba$. For oriented matroids, the commutation graph is known as the \emph{cocircuit graph} in the literature~\cite{Orientedmatroids}.

\begin{Thm}\label{commutationgraph}
Let $B$ be a finite geometric left regular band. Then $\Delta(\wh 0,\wh 1)$ is
homotopy equivalent to the clique complex $\Cliq(\Gamma(\mathscr M(B)))$ of the
commutation graph $\Gamma(\mathscr M(B))$ of the maximal elements of
$B\setminus \{1\}$.
\end{Thm}

Theorem~\ref{commutationgraph} can be used to give another proof that if a left
regular band $B$ is right hereditary, then $\Bbbk B$ is hereditary (\textit{cf.}
Section~\ref{ss:righthereditary}).
It also leads to an easy computation of the quiver.

\begin{Thm}\label{treethm}
    Let $B$ be a finite, right hereditary, left regular band and $\Bbbk$ a field.
    Then $\Bbbk B$ is hereditary.
    Its quiver has vertex set $\Lambda(B)$.
    The number of arrows from $X$ to $Y$ is zero if $X \not< Y$;
    otherwise, it is one less than the number of children of $y$ with support
    greater than or equal to $X$ (for any $y \in B$ with $\sigma(y) = Y$).
\end{Thm}


Theorem~\ref{treethm} covers the algebras of nearly all the left regular bands
considered by K.~Brown in~\cite{Brown1} (except the hyperplane monoids) and the
interval greedoid left regular bands of Bj\"orner~\cite{bjorner2}.

\subsection{Global dimension of a left regular band algebra}
\label{ss:gldim}

We now turn to applications of the higher degree cohomology spaces,
especially the first degree for which the higher degree cohomology vanishes.
This ``vanishing degree'' relates the \emph{Leray number} of the simplicial
complex with the \emph{global dimension} of the monoid algebra.

The \emph{global dimension} $\mathrm{gl.}\dim A$ of a finite dimensional
algebra $A$ over a field $\Bbbk$ is the smallest integer $n$ such
that $\Ext_A^{m}(S_1,S_2)=0$ for all $m>n$ and all simple $A$-modules
$S_1$ and $S_2$ (see~\cite{assem}).
An algebra $A$ has global dimension zero if and only if it is semisimple;
it has $\mathrm{gl.}\dim A\leq 1$ if and only if it is hereditary.

\begin{Thm}\label{globaldimension}
Let $B$ be a finite left regular band and $\Bbbk$ a field.  Then
\begin{align*}
    \mathrm{gl.}\dim \Bbbk B = \min \{n : \til H^n(\Delta(X,Y),\Bbbk)=0
                            \text{~for all~} X, Y\in\Lambda(B)
                            \text{~with~} X < Y\}.
\end{align*}
In particular, one has
$
    \mathrm{gl.}\dim \Bbbk B\leq m
$
where $m$ is the length of
the longest chain in $\Lambda(B)$.
\end{Thm}


As an immediate corollary, we obtain a characterization of those algebras $\Bbbk B$
that are hereditary in terms of the order complex of $B$.
\begin{Cor}
    The monoid algebra $\Bbbk B$ of a left regular band is hereditary if and
    only if each connected component of each simplicial complex $\Delta(X,Y)$,
    for $X < Y \in \Lambda(B)$, is acyclic.
\end{Cor}

\subsection{Leray numbers and an improved upper bound on global dimension}

We can improve on the upper bound from Theorem~\ref{globaldimension} using the
Leray number of a simplicial complex.
The \emph{$\Bbbk$-Leray number} of a simplicial complex $K$ with vertex set $V$
is
\begin{align*}
    L_\Bbbk(K) =\min\left\{d : \til H^i(K[W],\Bbbk) = 0
                \text{~for all~} i\geq d
                \text{~and all~} W\subseteq V \right\},
\end{align*}
where $K[W]$ denotes the subcomplex of $K$ consisting of all simplices whose
vertices belong to $W$.

\begin{Thm}\label{Leraybound}
Let $B$ be a finite left regular band and $\Bbbk$ a field.  Then $\mathrm{gl.}\dim
\Bbbk B$ is bounded above by the Leray number $L_\Bbbk(\Delta(B))$ of the order complex
of $B$.
\end{Thm}

Originally, interest in Leray numbers
came about because of connections with Helly-type
theorems~\cite{Wegner,BolandLek,Kalai2}:  the Leray number of a simplicial
complex provides an obstruction for realizing the complex as the nerve of
a collection of compact convex subsets of $\mathbb R^d$.
Leray numbers also play a role in combinatorial commutative algebra:
$L_\Bbbk(K)$ turns out to be the Castelnuovo-Mumford regularity of the
Stanley-Reisner ring of $K$ over $\Bbbk$~\cite{Kalai2};
or equivalently, $L_{\Bbbk}(K)+1$ is the regularity of the face ideal of
$K$~\cite{Kalai1}.

In section~\ref{ss:lerayFPCLRBs}, we give a new interpretation to the Leray
number of the clique complex of a graph in terms of non-commutative algebra.

\subsection{Leray numbers and free partially commutative left regular bands}
\label{ss:lerayFPCLRBs}

In this section, we prove that the global dimension of a free partially
commutative left regular band $B(\Gamma)$ is equal to the Leray number of
the clique complex of the graph $\Gamma$.

\begin{Thm}\label{partiallycommmain}
Let $\Gamma=(V,E)$ be a finite graph and $\Bbbk$ a commutative ring with unit.
Then, for $W\subsetneq U\subseteq V$ and $n\geq 1$, we have
\begin{align*}
    \Ext^n_{\Bbbk B(\Gamma)}(\Bbbk_U,\Bbbk_W)=\til H^{n-1}(\Cliq(\Gamma[U\setminus W]),\Bbbk).
\end{align*}
\end{Thm}
\begin{proof}
We present here the proof of the special case $W=\emptyset$ and $U=V$
(it turns out that the general case reduces to this case).
The maximal elements of $B(\Gamma)\setminus \{1\}$ are the elements of $V$.
The commutation graph for this set is exactly $\Gamma$.  Since $B(\Gamma)$ is
a geometric left regular band, we conclude $\Delta(\wh 0,\wh 1)$ is homotopy
equivalent to $\Cliq(\Gamma)$ by Theorem~\ref{commutationgraph}.  The theorem
now follows from Theorem~\ref{mainresult}.
\end{proof}

Our first corollary characterizes the free partially commutative left regular
bands with a hereditary $\Bbbk$-algebra. A graph is said to be \emph{chordal}
if it contains no induced cycle of length greater than $3$. It is known that
$L_\Bbbk(K)\leq 1$ if and only if $K$ is the clique complex of a chordal graph,
cf.~\cite{oldpaper,Wegner,BolandLek}.

\begin{Cor}
\label{FPLBgldim}
If $\Bbbk$ is a field and $\Gamma$ a finite graph, then the global dimension of
$\Bbbk B(\Gamma)$ is the $\Bbbk$-Leray number $L_\Bbbk(\Cliq(\Gamma))$. In particular,
$\Bbbk B(\Gamma)$ is hereditary if and only if $\Gamma$ is a chordal graph.
\end{Cor}

Our next corollary computes the quiver of the algebra of a free partially commutative left
regular band.

\begin{Cor}
Let $\Gamma=(V,E)$ be a finite graph.
The quiver of $\Bbbk B(\Gamma)$ has vertex set the power set of $V$. If $U\supsetneq
W$, then the number of arrows from $U$ to $W$ is one less than the number of
connected components of $\Gamma[U\setminus W]$.  There are no other arrows.
\end{Cor}


It is easy to see that if $\Gamma$ is \emph{triangle-free}, that is, has no
$3$-element cliques, then $\Cliq(\Gamma)=\Gamma$ and so $L_\Bbbk(\Cliq(\Gamma))=2$ unless
$\Gamma$ is a forest (in which case $\Gamma$ is chordal). This provides a natural infinite family of finite
dimensional algebras of global dimension $2$.

\subsection{CW left regular bands and minimal projective resolutions}

In this last section, we indicate how one can use topological properties of the
$\R$-poset of a left regular band $B$ to construct minimal projective
resolutions of the modules $\Bbbk_X$ (recall that these are all the simple
$\Bbbk B$-modules when $\Bbbk$ is a field, \emph{cf.} \S~\ref{s:mainresult}).

For this section, we assume that $B$ is a finite left regular band, but allow it to be a semigroup. To simplify notation below, we write $B_{\geq X}$ for $B[X, \wh 1]$, which we call the \emph{contraction} of $B$ to $X$ (it corresponds to the usual notion of contraction in hyperplane and oriented matroid theory).  We say that $B$ is \emph{connected} if each $\Delta(B_{\geq X})$ is connected; this is automatic if $B$ is a monoid since $1$ is a greatest element and hence $\Delta(B_{\geq X})$ is in fact contractible. In~\cite{newpaper} it is proved using  Theorem~\ref{mainresult} that $\Bbbk B$ is unital if and only if $B$ is connected and, moreover, the $\Delta(B_{\geq X})$ are all acyclic in this case.

\begin{Thm}\label{t:projectiveres}
Let $B$ be a connected left regular band and $\Bbbk$ a commutative ring with unit.  Then
the augmented simplicial chain complex
$C_\bullet(\Delta(B_{\geq X});\Bbbk)\xrightarrow{\,\,\varepsilon\,\,}\Bbbk_X$ is a $\Bbbk
B$-projective resolution.
%
\end{Thm}

Notice that $C_0(\Delta(B_{\geq X});\Bbbk)=\Bbbk B_{\geq X}$ and hence this resolution is almost never
minimal. We can improve upon this to get a minimal
resolution by taking advantage of the fact that the $\R$-posets of many of the
examples we have been discussing have very nice topology. For instance, for
real and complex hyperplane face monoids, the subposets $B[X,Y]$ are the face
posets of spheres
(see \S\S~\ref{ss:hyperplanemonoids}--\ref{ss:complexhyperplanemonoids}).

Following~\cite{bjornerCW}, we say that a poset $P$ is a
\emph{CW poset} if it is the face poset of a finite
regular CW complex; we denote the associated CW complex by $\Sigma(P)$.

Motivated by the many examples, we formulate the following definition.
We say that a left regular band $B$ is a \emph{CW left regular band} if
$B_{\geq X}$ is a connected CW poset for all $X \in \Lambda(B)$. For example, the set of
faces of a central or affine hyperplane arrangement or the set of covectors of an oriented
matroid (possibly affine) is a CW left regular band. The face poset of a $\mathrm{CAT}(0)$ cube complex~\cite{Wise,Chepoi} also admits a CW left regular band structure.  Tree space~\cite{BilleraHolmes}, sometimes called the tropical Grassmannian, is an example of a $\mathrm{CAT}(0)$ cube complex of combinatorial interest.

The enumerative combinatorics of a CW left regular band is like that of a hyperplane arrangement or an oriented matroid.  The support semilattice is Cohen-Macaulay and the $f$-vector is determined by the support semilattice via an analogue of Zaslavsky's theorem. See~\cite{newpaper} for details.

\begin{Thm}\label{t:cwlrb}
Let $B$ be a finite CW left regular band, $X \in \Lambda(B)$ and let $\Bbbk$ be
a field. Then the augmented cellular chain complex
$C_\bullet(\CW(B_{\geq X});\Bbbk)\to \Bbbk_X$
is the minimal projective resolution of $\Bbbk_{\geq X}$.

In particular, this applies to the case where $B$ is the face monoid of a (central or affine) real
or complex hyperplane arrangement, the set of covectors of an
oriented matroid or oriented interval greedoid, or the face monoid of a $\mathrm{CAT}(0)$ cube complex.
\end{Thm}

The key ingredient in the proof of Theorem~\ref{t:cwlrb} is that if $B$ is a CW
left regular band, then $B$ acts on the CW complex $\Sigma(B)$. Since
$\Sigma(B)$ is acyclic, its augmented chain complex becomes an acyclic chain
complex of $\Bbbk B$-modules, which turn out to be projective.

For CW left regular bands, we can also prove that $\Bbbk B$ is a Koszul
algebra, identify its Koszul dual algebra and its $\Ext$-algebra and describe a
presentation by quiver with relations. Below, $\Inc_{\Bbbk}(P)$ denotes the
incidence algebra of a poset $P$ over the ring $\Bbbk$.
The full details are presented in~\cite{newpaper}.

\begin{Thm}\label{t:Koszulfull}
Let $B$ be a finite CW left regular band and let $\Bbbk$ be a field.
Then $\Bbbk B$ is a Koszul algebra,
its Koszul dual algebra is $\Bbbk B^!\cong \Inc_{\Bbbk}(\Lambda(L)^{op})$
and its $\Ext$-algebra is $\Ext(\Bbbk B)\cong \Inc_{\Bbbk}(\Lambda(L))$.
%
\end{Thm}

%

\acknowledgements
\label{sec:ack}

The first author wishes to warmly thank the Center for Algorithmic and
Interactive Scientific Software, CCNY, CUNY for inviting him to be a Visiting
Professor during part of the research of this paper.
The second author was supported in part by NSERC and FQRNT.
The third author was supported in part by NSERC.
The first and third authors were supported in part by Binational Science
Foundation of Israel and the US (BSF) grant number 2012080.
Some of the research of this paper was done while the third author was at the
School of Mathematics and Statistics of Carleton University.

\bibliographystyle{abbrv}
\bibliography{lrbfpsac-references}

\begin{thebibliography}{10}

\bibitem{assem}
I.~Assem, D.~Simson, and A.~Skowro{\'n}ski.
\newblock {\em Elements of the representation theory of associative algebras.
  {V}ol. 1}, volume~65 of {\em London Mathematical Society Student Texts}.
\newblock Cambridge University Press, Cambridge, 2006.
\newblock Techniques of representation theory.

\bibitem{DiaconisAthan}
C.~A. Athanasiadis and P.~Diaconis.
\newblock Functions of random walks on hyperplane arrangements.
\newblock {\em Adv. in Appl. Math.}, 45(3):410--437, 2010.

\bibitem{BHR}
P.~Bidigare, P.~Hanlon, and D.~Rockmore.
\newblock A combinatorial description of the spectrum for the {T}setlin library
  and its generalization to hyperplane arrangements.
\newblock {\em Duke Math. J.}, 99(1):135--174, 1999.

\bibitem{BilleraHolmes}
L.~J. Billera, S.~P. Holmes, and K.~Vogtmann.
\newblock Geometry of the space of phylogenetic trees.
\newblock {\em Adv. in Appl. Math.}, 27(4):733--767, 2001.

\bibitem{bjornerCW}
A.~Bj{\"o}rner.
\newblock Posets, regular {CW} complexes and {B}ruhat order.
\newblock {\em European J. Combin.}, 5(1):7--16, 1984.

\bibitem{bjornersurvey}
A.~Bj{\"o}rner.
\newblock Topological methods.
\newblock In {\em Handbook of combinatorics, {V}ol.\ 1,\ 2}, pages 1819--1872.
  Elsevier, Amsterdam, 1995.

\bibitem{bjorner2}
A.~Bj{\"o}rner.
\newblock Random walks, arrangements, cell complexes, greedoids, and
  self-organizing libraries.
\newblock In {\em Building bridges}, volume~19 of {\em Bolyai Soc. Math.
  Stud.}, pages 165--203. Springer, Berlin, 2008.

\bibitem{Orientedmatroids}
A.~Bj{\"o}rner, M.~Las~Vergnas, B.~Sturmfels, N.~White, and G.~M. Ziegler.
\newblock {\em Oriented matroids}, volume~46 of {\em Encyclopedia of
  Mathematics and its Applications}.
\newblock Cambridge University Press, Cambridge, second edition, 1999.

\bibitem{complexstrat}
A.~Bj{\"o}rner and G.~M. Ziegler.
\newblock Combinatorial stratification of complex arrangements.
\newblock {\em J. Amer. Math. Soc.}, 5(1):105--149, 1992.

\bibitem{Bongartz}
K.~Bongartz.
\newblock Algebras and quadratic forms.
\newblock {\em J. London Math. Soc. (2)}, 28(3):461--469, 1983.

\bibitem{Brown1}
K.~S. Brown.
\newblock Semigroups, rings, and {M}arkov chains.
\newblock {\em J. Theoret. Probab.}, 13(3):871--938, 2000.

\bibitem{Brown2}
K.~S. Brown.
\newblock Semigroup and ring theoretical methods in probability.
\newblock In {\em Representations of finite dimensional algebras and related
  topics in {L}ie theory and geometry}, volume~40 of {\em Fields Inst.
  Commun.}, pages 3--26. Amer. Math. Soc., Providence, RI, 2004.

\bibitem{DiaconisBrown1}
K.~S. Brown and P.~Diaconis.
\newblock Random walks and hyperplane arrangements.
\newblock {\em Ann. Probab.}, 26(4):1813--1854, 1998.

\bibitem{CartierFoata}
P.~Cartier and D.~Foata.
\newblock {\em Probl\`emes combinatoires de commutation et r\'earrangements}.
\newblock Lecture Notes in Mathematics, No. 85. Springer-Verlag, Berlin, 1969.

\bibitem{Chepoi}
V.~Chepoi.
\newblock Graphs of some {${\rm CAT}(0)$} complexes.
\newblock {\em Adv. in Appl. Math.}, 24(2):125--179, 2000.

\bibitem{GrahamLRB}
F.~Chung and R.~Graham.
\newblock Edge flipping in graphs.
\newblock {\em Adv. in Appl. Math.}, 48(1):37--63, 2012.

\bibitem{Clifford}
A.~H. Clifford.
\newblock Semigroups admitting relative inverses.
\newblock {\em Ann. of Math. (2)}, 42:1037--1049, 1941.

\bibitem{tracebook}
V.~Diekert and G.~Rozenberg, editors.
\newblock {\em The book of traces}.
\newblock World Scientific Publishing Co. Inc., River Edge, NJ, 1995.

\bibitem{Green}
J.~A. Green.
\newblock On the structure of semigroups.
\newblock {\em Ann. of Math. (2)}, 54:163--172, 1951.

\bibitem{Kalai1}
G.~Kalai and R.~Meshulam.
\newblock Intersections of {L}eray complexes and regularity of monomial ideals.
\newblock {\em J. Combin. Theory Ser. A}, 113(7):1586--1592, 2006.

\bibitem{Kalai2}
G.~Kalai and R.~Meshulam.
\newblock Leray numbers of projections and a topological {H}elly-type theorem.
\newblock {\em J. Topol.}, 1(3):551--556, 2008.

\bibitem{BolandLek}
C.~G. Lekkerkerker and J.~C. Boland.
\newblock Representation of a finite graph by a set of intervals on the real
  line.
\newblock {\em Fund. Math.}, 51:45--64, 1962/1963.

\bibitem{oldpaper}
S.~Margolis, F.~Saliola, and B.~Steinberg.
\newblock Combinatorial topology and the global dimension of algebras arising
  in combinatorics, 2012.
\newblock \texttt{arXiv:1205.1159}.

\bibitem{qvariety}
S.~Margolis, F.~Saliola, and B.~Steinberg.
\newblock Semigroups embeddable in hyperplane face monoids, 2012.
\newblock \texttt{arXiv:1212.6683}, to appear in \emph{Semigroup Forum}.

\bibitem{newpaper}
S.~Margolis, F.~Saliola, and B.~Steinberg.
\newblock Cell complexes and projective resolutions for left regular bands,
  2014.
\newblock In preparation.

\bibitem{Rota}
G.-C. Rota.
\newblock On the foundations of combinatorial theory. {I}. {T}heory of
  {M}\"obius functions.
\newblock {\em Z. Wahrscheinlichkeitstheorie und Verw. Gebiete}, 2:340--368
  (1964), 1964.

\bibitem{Saliola}
F.~V. Saliola.
\newblock The quiver of the semigroup algebra of a left regular band.
\newblock {\em Internat. J. Algebra Comput.}, 17(8):1593--1610, 2007.

\bibitem{Saliolahyperplane}
F.~V. Saliola.
\newblock The face semigroup algebra of a hyperplane arrangement.
\newblock {\em Canad. J. Math.}, 61(4):904--929, 2009.

\bibitem{Wachs}
M.~L. Wachs.
\newblock Poset topology: tools and applications.
\newblock In {\em Geometric combinatorics}, volume~13 of {\em IAS/Park City
  Math. Ser.}, pages 497--615. Amer. Math. Soc., Providence, RI, 2007.

\bibitem{Wegner}
G.~Wegner.
\newblock {$d$}-collapsing and nerves of families of convex sets.
\newblock {\em Arch. Math. (Basel)}, 26:317--321, 1975.

\bibitem{Wise}
D.~T. Wise.
\newblock {\em From riches to raags: 3-manifolds, right-angled {A}rtin groups,
  and cubical geometry}, volume 117 of {\em CBMS Regional Conference Series in
  Mathematics}.
\newblock Published for the Conference Board of the Mathematical Sciences,
  Washington, DC, 2012.

\end{thebibliography}
\label{sec:biblio}

\end{document}